\documentclass[11pt]{amsart}
\usepackage[margin=1.00in]{geometry}
\usepackage{amsmath}
\usepackage{amsthm}
\usepackage{amsfonts}
\usepackage{tikz}
\usetikzlibrary{matrix,arrows,decorations.pathmorphing}
\usepackage{url}
\usepackage{hyperref}
\newcommand{\cH}{\mathrm{H}}
\newcommand{\sO}{\mathcal{O}}

\theoremstyle{plain}
\newtheorem{Theorem1}{Theorem}[section]

\newtheorem{Theorem2}[Theorem1]{Theorem}
\newtheorem{Shoulder}[Theorem1]{Lemma}
\newtheorem{Lemma2}[Theorem1]{Lemma}
\newtheorem{Lemma3}[Theorem1]{Lemma}
\newtheorem{proposition1}[Theorem1]{Proposition}
\newtheorem{fact}[Theorem1]{Fact}
\newtheorem{lemma4}[Theorem1]{Lemma}
\newtheorem{lemma5}[Theorem1]{Lemma}
\newtheorem{proposition2}[Theorem1]{Proposition}

\theoremstyle{definition}
\newtheorem{remark1}{Remark}
\newtheorem{remark3}[remark1]{Remark}
\newtheorem{remark4}[remark1]{Remark}

\title{An Enriques Classification Theorem for Surfaces in Positive Characteristic}
\author{Eugenia Ferrari}

\begin{document}
\maketitle
\begin{abstract}
We prove that a smooth projective surface $S$ over an algebraically closed field of characteristic $p>3$ is birational to an abelian surface if $P_1(S)=P_4(S)=1$ and $h^1(S,\sO_S)=2$.
\end{abstract}
\section{Introduction}
An often--tackled problem in algebraic geometry is that of characterizing projective varieties in terms of their birational invariants. In this frame, a classical result of Enriques states that a smooth complex surface $S$ with plurigenera $P_1(S)=P_4(S)=1$ and irregularity $h^1(S,\sO_S)=\dim S$ is birationally equivalent to an abelian surface (\cite{En1905}).\\
In the literature there are now several theorems birationally characterizing complex abelian varieties in terms of certain plurigenera and some other hypotheses. Among these results, Chen and Hacon proved in \cite{CH01} that a smooth complex projective variety $X$ with $P_1(X)=P_2(X)=1$ and $h^1(X,\sO_X)=\dim X$ is birational to an abelian variety.\\
What these works have in common is that they rely on the theory of generic--vanishing, and in particular on statements that are known to fail in positive characteristic (\cite{HK15}, \cite{Fi17}). Therefore it is still an open question whether and what kind of generalizations of Enriques' theorem hold when dealing with varieties defined over fields of positive characteristic.\\
When in characteristic zero, given a variety $X$, the dimension of $\cH^1(X,\sO_X)$ equals the dimension of $\mathrm{Alb}(X)$, the Albanese variety of $X$. In this situation, one can proceed to prove that the Albanese morphism gives a birational morphism between $X$ and $\mathrm{Alb}(X)$, provided that $h^1(X,\sO_X)=\dim X$. In positive characteristic $h^1(X,\sO_X)$ does not necessarily equal $\dim\mathrm{Alb}(X)$, so that if for a surface $S$ one fixes $h^1(S,\sO_S)=2$, then the Albanese variety of $S$ is not necessarily a surface.

While generalizations to higher dimension will be the topic of subsequent papers, the result proved here is

\begin{Theorem1}\label{mainthm}
Let $S$ be a smooth projective surface over an algebraically closed field of characteristic $p>3$. If
\begin{equation}
P_1(S)=P_4(S)=1,\qquad h^1(S,\sO_S)=2,
\end{equation}
then S is birational to an abelian surface.
\end{Theorem1}
The conditions $P_1(S)=P_4(S)=1$ imply that $P_2(S)=1$. The statement would be stronger if one could fix $P_1(S)=P_2(S)=1$ and make no requirement about $P_4(S)$. By the work done here, if one asks only for $P_1(S)=P_2(S)=1$ there would be only a few specific cases of surfaces that might cause the conclusion of the theorem to fail.\\
\\ 
The result is proven by considering the Kodaira dimension of $S$ and checking that those fixed invariants lead to a contradiction in all cases but when $\kappa(S)=0$ and $S$ is birational to an abelian surface.\\
Indeed, clearly, having a non-zero plurigenus implies that $\kappa(S)\neq -\infty$.\\
Moreover, $S$ cannot be of general type: the plurigenera are birational invariants and, from \cite[Corollary 1.8]{Ek88}, $P_2(S)\geq 2$ for a surface $S$ of general type.\\
Also, assume $\kappa(S)=0$. As observed in the proof of \cite[Theorem 1.2]{Li09}, if the characteristic is neither 2 nor 3 the $\mathrm{Pic}^0$ of a surface is reduced and $\Delta:=2h^1(S, \sO_S)-b_1$ is zero.\\
Therefore, looking at the table in \cite[Introduction and Preliminary Reductions]{BM77}, it is immediate that, since  $h^1(S, \sO_S)=2$, $S$ must be birational to an abelian surface.\\
$P_1(S)=1$ alone would have ruled out the cases of Enriques and hyperelliptic surfaces, but not K3 surfaces.\\
\\
From this discussion, either $\kappa(S)=1$ or $S$ is birational to an abelian surface. The work done  in the following sections is for ruling out the case $\kappa(S)=1$.

\subsection*{The condition $P_4(S)=1$.} As mentioned, assuming only $P_1(S)=P_2(S)=1$ would not allow to rule our a few  cases. The condition $P_4(S)=1$ is used to exclude one particular case of elliptic fibration onto a curve of genus zero when the characteristic of the base field is $5$ (see Remark~\ref{caratter5} and Remark~\ref{caratt5secondo}), and the case of an elliptic fibration over a curve of genus one  with exactly one multiple fibre which is not wild (see Remark~\ref{seilgenerecurvauno}).\\
In particular, in the first of the two cases mentioned, \cite[Lemma 3.3]{KU85} would imply that $\dim\mathrm{Alb}(S)=1$, therefore $\dim\mathrm{Alb}(S)<h^1(S,\sO_S)$, and this would not happen in characteristic zero. In the second case, i.e. the fibration onto a curve of genus one, by \cite[Lemma 3.4]{KU85} $\mathrm{Alb}(S)$ could be either a surface or a curve.
\subsection*{Characteristic $2$ and $3$.} The assumption that $p\geq 5$ is needed when $\kappa(S)=0$ to rule out the possibility of $S$ being birational to an hyperelliptic or a quasi--hyperelliptic surface, and when $\kappa(S)=1$. In this latter case, a part of what is done in Section~\ref{genere 1} goes through when $p=2,3$ (also considering the case of quasi-elliptic surfaces), and the computations made here allow only to say that, if such an $S$ exists, then it must be birational to a (quasi--)elliptic surface fibred over a curve $B$ of genus either $1$ or $0$. If $g(B)=1$, then the elliptic fibration has exactly one multiple fibre that is not wild, and if $g(B)=0$, then the fibration has either one or two multiple fibres, both wild.
\subsection*{Acknowledgements} The author is indebted to her advisor S. Tirabassi for pointing out the problem to her, for the many discussions, all the advice, and the constant encouragement. Also, the author would like to thank  F. Catanese, A. L. Knutsen and G. Scattareggia for insights and suggestions, and in particular A. L. Knutsen for reading a draft of this work and giving valuable advice. Thanks are due to T. Lundemo and M. R. Vodrup for pointing out misprints in a previous version of this paper. Part of this work was supported by the Meltzer Research Fund and by the project \textit{The Arithmetic of Derived Categories}, grant 261756 of the Research Council of Norway.

\section*{Conventions and notations}
$S$ is assumed to be a smooth projective surface over an algebraically closed field $k$ of characteristic $p>0$. We denote by $\omega_S$ the canonical bundle of $S$  and by $K_S$ a canonical divisor in $\omega_S$.\\
The Kodaira dimension of $S$ is denoted $\kappa(S)$ and for $i\in\mathbb{Z}$ $h^i(S,\cdot):=\dim \cH^i(S,\cdot)$. Also we denote by $\chi(S):=\sum_i(-1)^{i}h^{i}(S,\sO_S)$ the Euler characteristic of $S$.\\
The plurigenera of $S$ are $P_n(S):=\dim\cH^0(S,\omega_S^{\otimes n})$ for $n$ positive integer.

\section{Background and preliminary results} \label{genere 1}
As mentioned previously, a surface $S$ with $P_1(S)$, $P_2(S)$ and $h^1(S,\sO_S)$ fixed as from Theorem \ref{mainthm} must either be birational to an abelian variety or have Kodaira dimension $1$, and therefore the latter is the case which is to be studied and excluded.\\
Since the goal is to classify $S$ birationally having fixed those birational invariants, $S$ can be assumed to be a minimal surface. This assumption is helpful because for a minimal surface $S$ with $\kappa(S)=1$, in characteristic neither 2 nor 3, the Stein factorisation of the Iitaka fibration gives a relatively minimal elliptic fibration
\begin{equation}
f:\;S\longrightarrow B
\end{equation}
onto a non--singular curve $B$ (see for example \cite[Theorem 5.3]{Li12}).\\
The remainder of the work done here will deal with such an elliptic fibration, therefore in this section some basic facts about elliptic fibrations are recalled. More complete references for elliptic fibrations in characteristic $p$ are for example \cite[1. ]{BM77} and \cite[1.]{KU85}.\\
\\
Let $b_1,...b_r\in B$ be the finitely many points at which the fibre $f^{-1}(b_{\alpha})$ is multiple, that is to say:
\begin{equation}
f^{-1}(b_{\alpha})=m_{\alpha}P_{\alpha}
\end{equation}
with $m_{\alpha}\geq 2$ and $P_{\alpha}$ indecomposable of canonical type (recall that a curve $C=\sum n_iC_i$ is of canonical type if for all $i$ one has $(K_S\cdot C_i)=(C\cdot C_i)=0$).
Also let
\begin{equation}
R^1f_{\ast}\sO_S = L\oplus T \label{splittingR1}
\end{equation}
be the decomposition of $R^1f_{\ast}\sO_S$ into an invertible sheaf $L$ and a torsion sheaf $T$ supported over the points of the curve $B$ whose inverse image is a wild fibre.\\ 
The following theorem holds (\cite[Theorem 2]{BM77}):

\begin{Theorem2} \label{Kodairacanonical}
Let $f:\;S\longrightarrow B$ be a relatively minimal elliptic fibration and let $R^1f_{\ast}\sO_S = L\oplus T$. Then
\begin{equation}
\omega_S=f^{\ast}(L^{-1}\otimes \omega_B)\otimes \sO\big(\sum a_{\alpha}P_{\alpha}\big)
\end{equation}
where
\begin{itemize}
\item[a.] $m_{\alpha}P_{\alpha}$ are the multiple fibres;
\item[b.] $0\leq a_{\alpha}<m_{\alpha}$;
\item[c.] $a_{\alpha}= m_{\alpha}-1$ if $m_{\alpha}P_{\alpha}$ is not wild;
\item[d.] $\deg (L^{-1}\otimes\omega_B)= 2g(B)-2+\chi(\sO_S)+\mathrm{length}(T)$.
\end{itemize}
\end{Theorem2}
\subsection{Relations among numbers of the elliptic fibration}
To prove that an $S$ such as in Theorem~\ref{mainthm} is an abelian surface, one has to show that there cannot exist an elliptic fibration $f:\;S\longrightarrow B$. The first step for achieving this is to show that if such a fibration existed the genus of the base curve would be bounded by having fixed $h^1(S,\sO_S)$.

\begin{Shoulder} \label{Shoulder}
Let $f:\;S\longrightarrow B$ be a quasi--elliptic surface or an elliptic surface. Then
\begin{equation}
g(B)\leq \dim\cH^1(S,\sO_S)
\end{equation}
\end{Shoulder}
\begin{proof}
Indeed, one has a Leray spectral sequence
\begin{equation}
\mathrm{E}^{p,q}_2=\cH^p(B,R^qf_{\ast}\sO_S)\qquad \Longrightarrow \qquad E^{p+q}=\cH^{p+q}(S,\sO_S)
\end{equation}
By \cite[III, Proposition 8.1]{Ha77} the sheaves $R^qf_{\ast}\sO_S$ are trivial except possibly when $0\leq q\leq 2$, and so the page two of the above spectral sequence has zeroes except for a rectangle consisting of two objects in the horizontal direction and three in the vertical one. It follows that at page two there are no differentials between two nonzero vector spaces, and thus the sequence degenerates here and therefore $\cH^1(S,\sO_S)$ can be split as:
\begin{equation}\label{splitH1}
\begin{split}
\cH^1(S,\sO_S)&=\cH^0(B,R^1f_{\ast}\sO_S)\oplus \cH^1(B,f_{\ast}\sO_S)\\
&=\cH^0(B,L)\oplus \cH^0(B,T)\oplus \cH^1(B,\sO_B)
\end{split}
\end{equation}
and $\cH^1(B,\sO_B)$ has dimension $g(B)$ since $B$ is smooth. The conclusion follows.
\end{proof}
The following lemma provides some useful relations among numbers connected to objects on the base curve $B$.

\begin{Lemma2} \label{relations}
Let $f:\;S\longrightarrow B$ be a minimal elliptic surface with $h^1(S, \sO_S)=2$ and $P_1(S)=1$. Let $L$ and $T$ be as in (\ref{splittingR1}). Then 
\begin{enumerate}

\item[(i)] $h^0(B,L)-\deg L+g(B)=2$; \label{l1}
\item[(ii)] $\deg L = -\,h^0(B,T)\leq 0$;  \label{l2}
\item[(iii)] $h^1(B,L)=1 $; \label{l3}
\item[(iv)] $g(B)\leq 2$. \label{l4}
\end{enumerate}
\end{Lemma2}
\begin{proof}
Given those invariants, $\chi(S)=h^0(S,\sO_S)-h^1(S,\sO_S)+h^2(S,\sO_S)=1-2+1=0$. Therefore, from Theorem~\ref{Kodairacanonical} it follows that 
\begin{equation}
\deg L= -\mathrm{length}(T)=-h^0(B,T);
\end{equation}
and this proves (ii).\\
From Lemma \ref{Shoulder}, one immediately has (iv).\\
From (\ref{splitH1}) and (ii) one gets (i). By (i) and the Riemann--Roch Theorem for curves
\begin{equation}
h^0(B,L)-\deg L+g(B)=h^1(B,L)+1,
\end{equation}
one gets (iii).
\end{proof}

\subsection{An inequality for the plurigenera}
When, with the notation of Theorem~\ref{Kodairacanonical}, $\omega_S=\sO\Big(\sum a_{\alpha}P_{\alpha}\Big)$, one gets a lower bound for $P_2(S)$ depending on the number of the fibres that give rise to an $a_{\alpha}$ as big as possible (i.e. either multiple fibres that are not wild, or strange fibres: wild fibres with $a_{\alpha}=m_{\alpha}-1$, following the terminology of \cite{KU85}).

\begin{Lemma3}
\label{relationforp2}
Let $f:\;S\longrightarrow B$ be a minimal elliptic surface. Assume that, with the notation of Theorem~\ref{Kodairacanonical}, $\omega_S=\sO\Big(\sum a_{\alpha}P_{\alpha}\Big)$.\\
Let $I$ be the set of the indices $\alpha$ such that $a_{\alpha}=m_{\alpha}-1$. If $n,t\in\mathbb{N}_{>0}$ are such that $t\leq \frac{n}{2}$, then
\begin{equation}
P_n(S)\geq \vert I\vert\cdot t+1-g(B),
\end{equation}
where $\vert I\vert$ denotes the cardinality of $I$.
\end{Lemma3}
\begin{proof}
Assume that 
\begin{equation} \label{assumptionprooflemma}
h^0(S, \omega_S^{\otimes n})\geq h^0\Bigg(S,\sO\bigg(\sum_{\alpha\in I}tm_{\alpha}P_{\alpha}\bigg)\Bigg).
\end{equation}
Then 
\begin{align}
P_n(S)=h^0(S, \omega_S^{\otimes n})&\geq h^0\Bigg(S,\sO\bigg(\sum_{\alpha\in I}tm_{\alpha}P_{\alpha}\bigg)\Bigg)\nonumber\\
&= h^0\bigg(S, f^{\ast}\bigg(\sO\bigg(\sum_{\alpha\in I}tb_{\alpha}\bigg)\bigg)\bigg)\nonumber\\
&= h^0\bigg(B, \sO\bigg(\sum_{\alpha\in I} tb_{\alpha}\bigg)\bigg)\qquad \qquad\qquad \qquad(\mathrm{connected\; fibres})\nonumber\\
&= h^1\bigg(B, \sO\bigg(\sum_{\alpha\in I} tb_{\alpha}\bigg)\bigg)+\sum_{\alpha\in I} t+1-g(B)\nonumber
\end{align}
by Riemann--Roch for curves, thence the statement of the lemma.\\
\\
It remains to verify (\ref{assumptionprooflemma}). That is true if
$$nK_S\geq\sum_{\alpha\in I}n(m_{\alpha}-1)P_{\alpha}\geq \sum_{\alpha\in I}tm_{\alpha}P_{\alpha},$$
and in turn that holds if, for every $\alpha$, 
$$n(m_{\alpha}-1)\geq tm_{\alpha},$$
which, being $n>t$, is equivalent to
$$m_{\alpha}\geq 1+ \frac{t}{n-t}.$$
Since $m_{\alpha}\geq 2$, this latter inequality is satisfied if
$$\frac{t}{n-t}\leq 1,$$
that is, again by $n>t$, when $t\leq \frac{n}{2}$.
\end{proof}

\section{Proof of Theorem~\ref{mainthm}}
From what said up to now, in order to prove Theorem~\ref{mainthm} one has only to show that a minimal surface $S$ with $P_1(S)=P_4(S)=1$ and $h^1(S,\sO_S)=2$ and Kodaira dimension $1$ cannot have an elliptic fibration $f:\;S\longrightarrow B$ onto a curve of genus $g(B)\leq 2$.\\
\\
As a first reduction, $g(B)$ cannot be $2$ because of $P_2(S)=1$. Indeed, given an elliptic fibration $f:\;S\longrightarrow B$ with $P_2(S)=1$,
\begin{align}
1=h^0(S, \omega_S^{\otimes 2})&=h^0(S, f^{\ast}(L^{-1}\otimes \omega_B)^2\otimes \mathrm{eff.divisor})\qquad\qquad (\mathrm{Theorem}~\ref{Kodairacanonical})\nonumber\\
&\geq h^0(S, f^{\ast}(L^{-2}\otimes \omega_B^2))\nonumber\\
&=h^0(B, L^{-2}\otimes \omega_B^2)\nonumber\\
&\geq h^0(B, L^{-2}\otimes \omega_B^2)-h^1(B, L^{-2}\otimes \omega_B^2)\nonumber\\
&= 1-g(B) +\deg (L^{-2}\otimes \omega_B^2)\qquad\qquad\qquad (\mathrm{Riemann-Roch})\nonumber\\
&=1-g(B)+ 4g(B)-4-2\deg L\nonumber\\
&=3g(B)-3-2\deg L \nonumber\\
&=3\cdot 2-3-2\cdot 0 \nonumber\\
&=3\nonumber
\end{align}
where the second to last equality holds if $g(B)=2$ because of (i) of Lemma~\ref{relations}: it must be $h^0(B,L)-\deg L=0$, and by (ii) of Lemma~\ref{relations} both $h^0(B,L)$ and $-\deg L$ are non--negative, therefore $\deg L=0$.\\
\\
The following sections deal with the remaining cases $g(B)=0$ and $g(B)=1$. In both these cases the relations of Lemma~\ref{relations} together with the formulas in Theorem~\ref{Kodairacanonical} allow to write $\omega_S$ as sheaf associated to a particular effective divisor coming from the multiple fibres of the elliptic fibration.

\subsection{Genus of the base curve equals one} 
Assume to have an elliptic fibration $f:\;S\longrightarrow B$, where $S$ has the birational invariants fixed as in Theorem~\ref{mainthm}. The purpose of this section is to show that the genus of $B$ cannot be one. If $g(B)=1$, by (i) of Lemma~\ref{relations}, one gets that either $h^0(B,L)=0$ and $\deg L=-1$ or  $h^0(B,L)=1$ and $\deg L=0$.\\
The first case can be written off because of $P_2(S)=1$ in a similar fashion to what has been done to rule out $g(B)=2$. Indeed, again by Theorem~\ref{Kodairacanonical}, projection formula and the Riemann--Roch Theorem for curves:
\begin{equation}
1=h^0(S, \omega_S^{\otimes 2})\geq 3g(B)-3-2\deg L =3\cdot 1-3-2\cdot(-1)=2.
\end{equation}

So, assume $h^0(B,L)=1$ and $\deg L=0$. It follows that $L$ must be $\sO_B$ because of \cite[IV, Lemma 1.2]{Ha77}. Since $h^0(B,T)=-\deg L=0$ and $T$ is torsion, it follows that $R^1f_{\ast}\sO_S=\sO_B$ and that there are no wild fibres.\\
By Theorem~\ref{Kodairacanonical},
\begin{equation}
\omega_S=\sO\Big(\sum (m_{\alpha}-1)P_{\alpha}\Big).
\end{equation} 
If there were at least two multiple fibres, Lemma~\ref{relationforp2} would imply that $P_2(S)\geq 2$. If there were no multiple fibres the canonical bundle of $S$ would be trivial, so $\kappa(S)=0$. 

\begin{remark1} \label{seilgenerecurvauno}
Without asking $P_4(S)=1$, the only case left open for $g(B)=1$ is when the fibration has exactly one multiple fibre which is not wild. Then Lemma~\ref{relationforp2} yields $P_n(S)\geq 2$ for $n\in\mathbb{N}_{\geq 4}$.
\end{remark1}

\subsection{Genus of the base curve equals zero}
From what previously done, a surface $S$ which satisfies the conditions of Theorem~\ref{mainthm}, is either birational to an abelian surface or it has $\kappa(S)=1$ and there is an elliptic fibration $f:\;S\longrightarrow B$ onto a smooth curve $B$ with $g(B)=0$. Therefore, Therem~\ref{mainthm} is proved on condition of proving the following:

\begin{proposition1} \label{proposizionepergenerezero}
Let $S$ be a minimal surface over an algebraically closed field of characteristic $p>5$ with $P_1(S)=P_2(S)=1$, $h^1(S,\sO_S)=2$ and Kodaira dimension 1. Then the elliptic fibration of $S$ cannot be onto a curve of genus 0.
\end{proposition1}
Having fixed $h^1(S,\sO_S)$, $P_1(S)$ and characteristic $p>3$, by (i) of Lemma~\ref{relations} there would be three possible cases:
\begin{enumerate}
\item $h^0(B,L)=2$, $\deg L=0$. By \cite[IV, Lemma 1.2]{Ha77}, $L\simeq \sO_{B}$, but then $h^0(B,L)=1$, contradiction.

\item $h^0(B,L)=1$, $\deg L=-1$. Impossible, an effective divisor has positive degree.
\item $h^0(B,L)=0$, $\deg L=-2$. Since the genus is $0$, $L\simeq \sO(-2)\simeq \omega_B$. Therefore, by Theorem~\ref{Kodairacanonical},
\begin{equation} \label{canonicalgenus0}
\omega_S=\sO\Big(\sum a_{\alpha}P_{\alpha}\Big)
\end{equation}
\end{enumerate}
So, if such a surface as in the proposition existed, one would be in the last case.\\
\\
If no fibre appeared in the expression of the canonical bundle, then the latter would be trivial, and the Kodaira dimension would not be 1. If there was at least one multiple fibre not wild, then, by Lemma~\ref{relationforp2}, $P_2(S)\geq 2$, contradicting $P_2(S)=1$.\\
By (ii) of Lemma~\ref{relations} and $\deg L=-2$, there can be at most two wild fibres. So it remains to exclude the two cases when there is exactly one multiple fibre which is wild and when there are exactly two multiple fibres, both wild. These two cases are addressed by the next two lemmas. Both of them are based on the following fact, used many times in \cite{KU85}:

\begin{fact} \label{factcohomology}
Let $f:\;S\longrightarrow B$ be an elliptic surface. Then, using the notation already introduced,
\begin{equation}
h^0(m_{\alpha}P_{\alpha}, \sO_{m_{\alpha}P_{\alpha}})=1+h^0(B, T\otimes_{\sO_B}k(b_{\alpha}))
\end{equation}
\end{fact}
\begin{proof}
As mentioned in the proof of \cite[Lemma 1.2]{KU85}, one has $$h^0(m_{\alpha}P_{\alpha}, \sO_{m_{\alpha}P_{\alpha}})=h^1(m_{\alpha}P_{\alpha}, \sO_{m_{\alpha}P_{\alpha}}).$$ Indeed, by 
\begin{equation}
0 \longrightarrow \sO_S(-m_{\alpha}P_{\alpha}) \longrightarrow \sO_S \longrightarrow \sO_{m_{\alpha}P_{\alpha}} \longrightarrow 0
\end{equation}
one gets $\chi(\sO_S(-m_{\alpha}P_{\alpha}))+\chi(\sO_{m_{\alpha}P_{\alpha}})=\chi(\sO_S)$, and the claim $\chi(\sO_{m_{\alpha}P_{\alpha}})=0$ follows from the Riemann--Roch theorem for surfaces and the fact that the fibres of $f$ are curves of canonical type:
\begin{equation}
\chi(\sO_S(-m_{\alpha}P_{\alpha}))=\chi(\sO_S)+\frac{-m_{\alpha}P_{\alpha}\cdot(-m_{\alpha}P_{\alpha}-K_{S})}{2}=\chi(\sO_S).
\end{equation}
Then the statement of the fact holds because of the equalities
\begin{align*}
h^0(m_{\alpha}P_{\alpha}, \sO_{m_{\alpha}P_{\alpha}})&=h^1(m_{\alpha}P_{\alpha}, \sO_{m_{\alpha}P_{\alpha}})=h^0(B, R^1f_{\ast}\sO_S\otimes k(b_{\alpha}))\\
&=h^0(B, L\otimes k(b_{\alpha}))+ h^0(B, T\otimes k(b_{\alpha}))=1+h^0(B, T\otimes k(b_{\alpha})),
\end{align*}
where the second equality holds by cohomology and base change (\cite[Corollary 3]{Mu12}).
\end{proof}

\begin{lemma4}
Let $S$ be a minimal surface over an algebraically closed field of characteristic $p>3$ with $P_1(S)=P_2(S)=1$, $h^1(S,\sO_S)=2$ and Kodaira dimension 1. Assume that the elliptic fibration of $S$ is onto a curve $B$ of genus 0. Then it is not possible to have exactly 2 multiple fibres, both wild.
\end{lemma4} 
\begin{proof}
Let $m_1P_1$ and $m_2P_2$ be the two wild fibres. By Fact~\ref{factcohomology}, $h^0(m_iP_{i}, \sO_{m_{i}P_{i}})=2$. So one can apply \cite[Lemma 2.4 (i)]{KU85} and see that

\begin{equation}
a_i=\left\{
\begin{array}{c}
m_i-1\\
\mathrm{or}\\
m_i-\nu_i-1,
\end{array} \right.
\end{equation}
where $\nu_i$ are positive integers defined in \cite[1. Preliminaries]{KU85}. If for at least one of the two fibres the first equality held, with $P_2(S)$ one would get a contradiction applying Lemma~\ref{relationforp2}. So, one can assume that for both fibres
\begin{equation}
a_i=p^{\delta_i}\nu_i-\nu_i-1,
\end{equation}
where $m_i=p^{\delta_i}\nu_i$ because of \cite[(1.6)]{KU85}, and $\delta_i\geq 1$ because $a_i\geq 0$. It is worth noticing that at least one of the two $a_i$ must be strictly positive, otherwise by (\ref{canonicalgenus0}) $\omega_S$ would be trivial, impossible because of $\kappa(S)=1$.\\
The goal here is to reach a contradiction by showing that this would lead to $P_2(S)\geq 2$.\\

First, for an $i$ such that $a_i\neq 0$,
$$2a_i=2(p^{\delta_i}\nu_i-\nu_i-1)\geq p^{\delta_i}\nu_i=m_i.$$
Indeed, that inequality is equivalent to:
\begin{equation*}
2\geq \frac{p^{\delta_i}\nu_i}{p^{\delta_i}\nu_i-\nu_i-1}=\frac{p^{\delta_i}\nu_i-\nu_i-1+\nu_i+1}{p^{\delta_i}\nu_i-\nu_i-1}=1+\frac{\nu_i+1}{p^{\delta_i}\nu_i-\nu_i-1},
\end{equation*}
which is true if $\frac{\nu_i+1}{p^{\delta_i}\nu_i-\nu_i-1}\leq 1$, that is: $p^{\delta_i}\nu_i-2\nu_i-2\geq 0$. Since $\nu_i,\delta_i\geq 1$, that equality is satisfied if $p^{\delta_i}-4\geq 0$ holds. Since the characteristic is neither $2$ nor $3$, that inequality is satisfied.\\

Having proved above that $2a_i\geq m_i$,  for a fixed $i$ one has that
\begin{equation}
2K_S\geq 2a_iP_i\geq m_iP_i
\end{equation}
and so
\begin{equation}
h^0(S, \omega_S^{\otimes 2})\geq h^0(S, \sO(m_iP_i))
\end{equation}
then, similarly to the proof of Lemma~\ref{relationforp2},
\begin{align}
P_{2}(S)=h^0(S, \omega_S^{\otimes 2})&\geq h^0(S, \sO(m_iP_i))\nonumber\\
&= h^0(S, f^{\ast}(\sO(b_{i}))\nonumber\\
&= h^0(B, \sO(b_{i}))\nonumber\\
&= h^1(B, \sO(b_{i}))+1+1-g(B)\geq 2\nonumber
\end{align}
which contradicts $P_2(S)=1$.
\end{proof}
Finally, to conclude the proof of Proposition~\ref{proposizionepergenerezero} one has to exclude the case of an elliptic fibration with exactly one multiple fibre which is wild, and the following lemma deals with that case.

\begin{lemma5} \label{onewildlemma}
Let $S$ be a minimal surface over an algebraically closed field of characteristic $p>5$ with $P_1(S)=P_2(S)=1$, $h^1(S,\sO_S)=2$ and Kodaira dimension 1. Assume that the elliptic fibration of $S$ is onto a curve $B$ of genus 0. Then it is not possible to have exactly 1 multiple fibre which is wild.
\end{lemma5}  
\begin{proof}
Let $mP$ be the wild fibre over the point $b$ of $B$, $a\geq 1$ the coefficient of $P$ in (\ref{canonicalgenus0}). Then
\begin{enumerate}
\item The hypotheses of \cite[Corollary 4.2]{KU85} are satisfied. Therefore $m=p^{\delta}$ (with $\delta$ positive integer) and $\nu=1$.
\item Since $T$ has length $2$ and is supported only on $b$, by Claim~\ref{factcohomology} $h^0(mP,\sO_{mP})=1+2=3$. 
\end{enumerate}
By these two facts, one can apply \cite[Lemma 2.4 (ii)]{KU85} and get the following possibilities:
\begin{equation}
a=\left\{
\begin{array}{c}
p^{\delta}-1\\
p^{\delta}-2\\
p^{\delta}-3\\
p^{\delta}-p-2
\end{array} \right.
\end{equation}
\begin{itemize}
\item[(i)] If $a=p^{\delta}-p-2$. Since $a\geq 0$, $\delta\geq 2$. One can use the strategy of Lemma~\ref{relationforp2} and get a contradiction with $P_2(S)=1$ if 
$$2(p^{\delta}-p-2)\geq p^{\delta},$$
that is
$$2\geq \frac{p^{\delta}}{p^{\delta}-p-2}=\frac{p^{\delta}-p-2+p+2}{p^{\delta}-p-2}=1+\frac{p+2}{p^{\delta}-p-2}.$$
That is true if $\frac{p+2}{p^{\delta}-p-2}\leq 1$, equivalently:
$$p^{\delta}-2p-4\geq 0.$$
Since $\delta\geq 2$, this last inequality is satisfied when the characteristic $p$ is neither 2 nor 3. 
\item[(ii)] If $a=p^{\delta}-s$ with $s=1,2$ or $3$. The contradiction with $P_2(S)=1$  can be obtained using the strategy of Lemma~\ref{relationforp2} if
$$2(p^{\delta}-s)\geq p^{\delta},$$
equivalently:
$$2\geq \frac{p^{\delta}}{p^{\delta}-s}=\frac{p^{\delta}-s+s}{p^{\delta}-s}=1+\frac{s}{p^{\delta}-s}.$$
Therefore one needs $\frac{s}{p^{\delta}-s}\leq 1$, that is $p^{\delta}-2s\geq0$. This is true if $s=1$ or $2$, while if $s=3$ it is false in a handful of cases if $p=2$ or $3$, and also when $p=5$ and $\delta=1$.
\end{itemize}
\end{proof}

\begin{remark3} \label{caratter5} The proof of Lemma~\ref{onewildlemma} does not go through in characteristic $5$ only  because of the possible existence of a fibration with one wild fibre $5P$ of multiplicity $5$, order $1$, and such that  $\omega_S=\sO(2P)$.
\end{remark3}
Actually, modifying point (ii) in the proof of Lemma~\ref{onewildlemma} by checking for which integers $n$ one has
\begin{equation}
n(5^{\delta}-s)\geq 5^{\delta},
\end{equation}
one gets the following remark:

\begin{remark4} \label{caratt5secondo}
Let $S$ be a minimal surface over an algebraically closed field of characteristic $5$ with $P_1(S)=P_n(S)=1$ for a fixed $n\in\mathbb{N}_{\geq 3}$, $h^1(S,\sO_S)=2$ and Kodaira dimension 1. Assume that the elliptic fibration of $S$ is onto a curve $B$ of genus 1. Then it is not possible to have exactly 1 multiple fibre which is wild.
\end{remark4}
 Having proved Proposition~\ref{proposizionepergenerezero} and having recovered the case of $p=5$ in Remark~\ref{caratter5} and Remark~\ref{caratt5secondo}, Theorem~\ref{mainthm} holds.\\
As a final remark, one could have proved Proposition~\ref{proposizionepergenerezero} in a slightly different fashion by following the computations in the proof of \cite[Theorem 5.2]{KU85} and specializing them to the case at hand, thus getting the statement of Proposition~\ref{KUrisolvetutto} (which is stronger than Proposition~\ref{proposizionepergenerezero}).\\
The case of elliptic surface in characteristic 5 that did not allow to state Lemma~\ref{onewildlemma} for $p\geq 5$ is the same that forces in Proposition~\ref{KUrisolvetutto} to distinguish between $p\geq 5$ and $p\geq 7$, and it is the same case that has been dealt with in Remark~\ref{caratter5} and Remark~\ref{caratt5secondo}.

\begin{proposition2} \label{KUrisolvetutto}
Let  $f:\;S\longrightarrow B$ be an algebraic elliptic surface in characteristic $p\geq 5$ with $g(B)=0$, $\kappa(S)=1$, $P_1(S)=1$, $h^1(S,\sO_S)=2$ then $\vert mK_S\vert$ gives the unique structure of the elliptic surface when $m\geq 3$.
If $p\geq 7$,then the same holds for $m\geq 2$. 
\end{proposition2}
\begin{proof}
Since $g(B)=0$, $\chi(\sO_S)=0$ and (as seen in (ii) of Lemma~\ref{relations}) $t=h^0(B,T)=-\deg L=2$ , one is in the situation (III) of the proof of \cite[Theorem 5.2]{KU85}, that is, one gets that, for $m\in\mathbb{N}$, $\vert mK_S\vert$ gives the unique structure of the elliptic surface if, with the notation of Theorem~\ref{Kodairacanonical},
\begin{equation}
\mathfrak{D}:=\sum_{\alpha}\Bigg\lfloor \frac{m\,a_{\alpha}}{m_{\alpha}} \Bigg\rfloor \geq 1. \label{ineqproofKU}
\end{equation}

It is directly stated and proved in \cite{KU85} that if the elliptic fibration $f:\;S\longrightarrow B$ has at least one tame fibre then (\ref{ineqproofKU}) is satisfied for $m\geq 2$ (and by their computations, the same is true if there is at least one wild fibre of strange type, i.e a wild fibre with $a_{\alpha}=m_{\alpha}-1$). So one can reduce to the case where the only multiple fibres are not tame and not wild of strange type. By $h^0(B,T)=2$, there are at most two such fibres.\\
\\
If there is exactly one wild fibre not of strange type then, following \cite{KU85}, only three cases are possible:
\begin{itemize}
\item[(i)] $a_1=m_1-\nu_1-1$,
\item[(ii)] $a_1=m_1-2\nu_1-1$,
\item[(iii)] $a_1=m_1-(p+1)\nu_1-1$.
\end{itemize}
As shown in \cite{KU85}, in case (i) $\mathfrak{D}=\lfloor m(1-\frac{1}{p^{\gamma}}-\frac{1}{p^{\gamma}\nu_1})\rfloor$, with $\gamma$, $\nu\geq 1$. Here, taking $p\geq 5$, one has $\mathfrak{D}\geq\lfloor m(1-\frac{1}{5}-\frac{1}{5})\rfloor = \lfloor m\frac{3}{5}\rfloor$, and so (\ref{ineqproofKU}) is satisfied for $m\geq 2$.\\
Similarly, in case (ii) $\mathfrak{D}=\lfloor m(1-\frac{2}{p^{\gamma}}-\frac{1}{p^{\gamma}\nu_1})\rfloor$, with $\gamma$, $\nu\geq 1$. If $p\geq 5$, then $\mathfrak{D}\geq\lfloor m(1-\frac{2}{5}-\frac{1}{5})\rfloor = \lfloor m\frac{2}{5}\rfloor$, and so (\ref{ineqproofKU}) is satisfied for $m\geq 3$. If $p\geq 7$, then $\mathfrak{D}\geq\lfloor m(1-\frac{2}{7}-\frac{1}{7})\rfloor = \lfloor m\frac{4}{7}\rfloor$, and so (\ref{ineqproofKU}) is satisfied for $m\geq 2$.\\
In case (iii), since $m_1=p^{\gamma}\nu_1$,
\begin{equation}
\mathfrak{D}=\Bigg\lfloor m\,\frac{p^{\gamma}\nu_1-(p+1)\nu_1-1}{p^{\gamma}\nu_1}\Bigg\rfloor=\Bigg\lfloor m\Bigg(1-\frac{1}{p^{\gamma-1}}-\frac{1}{p^{\gamma}}-\frac{1}{p^{\gamma}\nu_1}\Bigg)\Bigg\rfloor
\end{equation}
with $\gamma\geq 2$ because of $a_1>0$. If $p\geq 5$,
\begin{equation}
\mathfrak{D}\geq \Bigg\lfloor m\Bigg(1-\frac{1}{5}-\frac{1}{5^2}-\frac{1}{5^{2}}\Bigg)\Bigg\rfloor=\Bigg\lfloor m\,\frac{18}{25}\Bigg\rfloor
\end{equation}
and therefore (\ref{ineqproofKU}) is satisfied for $m\geq 2$.\\
\\
If there are exactly two wild fibres, it is shown directly in \cite{KU85} that (\ref{ineqproofKU}) holds when $m\geq 4$, but since in that case the $a_{\alpha}$ are exactly those of case (i) just above, taking $p\geq 5$ (\ref{ineqproofKU}) is satisfied for $m\geq 2$. 
\end{proof}

\textsc{Matematisk Institutt, Universitetet i Bergen}\\
\indent\textit{E--mail address}: \texttt{Eugenia.Ferrari@uib.no}

\begin{thebibliography}{[ACGH]}
\bibitem[BM77]{BM77}
\textsc{E. Bombieri, D. Mumford}, \emph{Enriques' Classification of Surfaces in Char. p, II}, Complex analysis and algebraic geometry, pp. 23--42. Iwanami Shoten, Tokyo, 1977.
\bibitem[CH01]{CH01}
\textsc{J.A. Chen, C.D. Hacon}, \emph{Characterization of abelian varieties}, Invent. Math. 143, pp. 435–-447, 2001.
\bibitem[Ek88]{Ek88}
\textsc{T. Ekedahl}, \emph{Canonical models of surfaces of general type in positive characteristic}, Inst. Hautes \'Etudes Sci. Publ. Math. No. 67, pp. 97--144, 1988.
\bibitem[En1905]{En1905}
\textsc{F. Enriques}, \emph{Sulle superficie algebriche che ammettono un gruppo continuo di trasformazioni birazionali in se stesse}, Rendic. Circolo Mat. di Palermo, 20, pp. 61--72, 1905.
\bibitem[Fi17]{Fi17}
\textsc{S. Filippazzi}, \emph{Generic vanishing fails for surfaces in positive characteristic}, Boll. Unione Mat. Ital., 2017, DOI 10.1007/s40574--017--0120--6. 
\bibitem[Ha77]{Ha77}
\textsc{R. Hartshorne}, \emph{Algebraic Geometry}, Graduate Texts in Mathematics No. 52,
Springer-Verlag, New York--Heidelberg, 1977.
\bibitem[HK15]{HK15} \textsc{C.D. Hacon, S.J. Kov\'acs}, \emph{Generic vanishing fails for singular varieties and in characteristic $p>0$}, Recent Advances in Algebraic Geometry, London Math. Soc. Lecture Note Ser. 417, pp. 240--253, Cambridge Univ. Press, Cambridge, 2015. 
\bibitem[KU85]{KU85}
\textsc{T. Katsura, K. Ueno}, \emph{On elliptic surfaces in characteristic p}, Math. Ann. 272, pp. 291--330, 1985.
\bibitem[Li09]{Li09}
\textsc{C. Liedtke}, \emph{A note on non-reduced Picard schemes}, J. Pure Appl. Algebra 213, pp. 737--741, 2009.
\bibitem[Li12]{Li12}
\textsc{C. Liedtke}, \emph{Algebraic surfaces in positive characteristic}, 2012, accessed at $<$https://www.semanticscholar.org/paper/Algebraic-Surfaces-in-Positive-Characteristic-Liedtke/3fdf5c486f51288c93baefa380035804b022d702$>$ on the 12 October 2017.
\bibitem[Mu12]{Mu12}
\textsc{D. Mumford}, \emph{Abelian Varieties}, with appendices by C. P. Ramanujam and Yuri Manin. Corrected reprint of the second (1974) edition. Tata Institute of Fundamental Research Studies in Mathematics, 5. Published for the Tata Institute of Fundamental Research, Bombay; by Hindustan Book Agency, New Delhi, 2014.
\end{thebibliography}
\end{document}